\documentclass[12pt]{amsart}

\usepackage{amsthm,amsmath,amssymb,amsfonts,color, mathtools}
\usepackage{mathrsfs}
\usepackage{tikz}
\usepackage{enumerate}

\newtheorem{thm}{Theorem}[section]
\newtheorem{lem}[thm]{Lemma}
\theoremstyle{definition}

\newtheorem{exa}[thm]{Example}

\newcommand{\Z}{\mathbb{Z}}
\newcommand{\Q}{\mathbb{Q}}
\newcommand{\C}{\mathbb{C}}
\newcommand{\B}{\mathbf{B}}

\DeclareMathOperator{\Aut}{Aut}

\begin{document}	
\title[Zeta functions for table algebras and fusion rings]{Zeta functions for table algebras and fusion rings with irrational-valued characters}
\author[ ]{Angelica Babei${^*}$ and Allen Herman$^{\dagger}$}
\address{Department of Mathematics $\&$ Statistics, McMaster University, Hamilton Hall, 1280 Main Street West, Hamilton, ON,
L8S 4K1, Canada }\email{babeiangelica@gmail.com}
\address{Department of Mathematics and Statistics, University of Regina,
  Regina, Saskatchewan S4S 0A2, Canada}\email{allen.herman@uregina.ca}

   \thanks{$^{*}$ Department of Mathematics $\&$ Statistics, McMaster University, Hamilton Hall, 1280 Main Street West, Hamilton, ON,
L8S 4K1, Canada}
 \thanks{$^{\dagger}$  Department of Mathematics and Statistics, University of Regina,
  Regina, Saskatchewan S4S 0A2, Canada} 
\thanks{$^{\dagger}$ The work of the second author is supported by an NSERC Discovery Grant.} 
\keywords{Zeta functions, orders, association schemes, integral adjacency algebras, table algebras, fusion rings}
\subjclass{Primary: 11S45. Secondary: 11R54, 05E30.}

\begin{abstract}	
We calculate ideal zeta functions for certain orders of rank $3$ defined by standard integral table algebras and integral fusion rings that have irrational-valued irreducible characters.  The calculations are obtained from explicit calculations of zeta integrals.   
\end{abstract}

\maketitle

\section{Introduction}

\medskip
Let $\Lambda$ be an order in a semisimple $\Q$-algebra $A$, and let $M$ be a $\Lambda$-lattice.  In \cite{Solomon77}, Solomon defined a zeta function $$ \zeta_{\Lambda}(M;s) =  \sum_{n \ge 1} a_n n^{-s} $$ 
where $a_n$ is the number of $\Lambda$-sublattices of $M$ of index $n$ for all $n \in \Z^+$.  In the case $M=\Lambda$, $\zeta_{\Lambda}(s) := \zeta_{\Lambda}(\Lambda;s)$ is called the {\it zeta function of $\Lambda$}, and is the generating function for the sequence $\{ a_n \}$ that counts the number of ideals of index $n$.

In \cite{Solomon77}, Solomon gave a calculation of $\zeta_{\Lambda}(s)$ in the case where $\Lambda = \Z G$ is an integral group ring of a finite cyclic group of prime order $p$.  Work of Hironaka (\cite{Hironaka81}, \cite{Hironaka85}) and Takegahara (\cite{Takegahara87}) in the 1980s produced the calculations of $\zeta_{\Z G}(s)$ where $G$ is any abelian of order $pq$, where $p$ and $q$ are not-necessarily distinct primes.  The same elementary approach was used by Hanaki and Hirasaka for integral adjacency rings of association schemes that have prime order or rank $2$ \cite{Hanaki-Hirasaka2016}, and by Hirasaka and Oh for quotient polynomial rings of the form $\Z[x]/(x-k)(x-a)(x-b)$ for $k,a,b \in \Z$ \cite{Hirasaka-Oh2018}.  In \cite{BH2022}, the authors applied 
a formula involving local zeta integrals due to Bushnell and Reiner in \cite{Bushnell-Reiner1980} to give explicit calculations of $\zeta_{\Z B}(s)$, where $\Z B$ was the integral adjacency ring of certain small association schemes with rational character tables. 

In this paper, we improve the methods of \cite{BH2022} to include cases where $\C \B$ has an irrational character table with respect to the $\Z$-basis $\B$.  This allows us to compute zeta functions of $\Lambda=\Z \B$ in more situations.  We illustrate these techniques by calculating  $p$-local zeta functions for the integral adjacency rings of association schemes of rank $3$ that correspond to doubly regular tournaments or conference graphs whose order $n$ has $p$-valuation $1$ or $3$ at an odd prime $p$, and the global zeta functions of all categorifiable integral fusion rings of rank $3$.

\section{Orders defined by table algebras and fusion rings}

An  {\it integral table algebra} $(A,\B)$ is a finite-dimensional algebra $A$ with (skew-linear) involution $*$ over $\C$ whose defining basis $\B = \{1=b_0, b_1, \dots, b_d\}$ contains $1$, is $*$-closed, admits non-negative integer structure constants $\{ \lambda_{ijk} : 0 \le i,j,k \le d \}$ satisfying 
$$b_i b_j = \sum_{k=0}^d \lambda_{ijk} b_k,$$ 
and satisfies the {\it pseudo-inverse} condition: for all $b_i \in B$, $b_i^*$  is the unique $b_j \in \B$ such that $\lambda_{ij0} > 0$, and furthermore if we denote the index of this $b_j$ by $i^*$, then $\lambda_{ii^*0}=\lambda_{i^*i0}$. 

Examples of integral table algebras include the group algebras of finite groups $G$ (whose skew-linear involution is given by $g^*=g^{-1}$ for all $g \in G$), the adjacency algebras (a.k.a.~Bose-Mesner algebras) of association schemes (for which $\B$ can be taken to be the standard basis of adjacency of adjacency matrices of the association scheme and the involution is the conjugate transpose), and the complexification of an integral fusion ring (which is precisely a commutative integral table algebra with the additional property $\lambda_{ii^*0}=\lambda_{i^*i0}=1$.)  When the algebra is commutative, the fact that the basis elements $b_i \in \B$ are represented by non-negative integer matrices in the regular representation implies that the largest eigenvalue $k_i > 0$ of each matrix will be obtained from a common Perron-Frobenius eigenvector, and thus the map $\delta(b_i)=k_i$ extends to a linear character of $A$.  This unique irreducible character that takes positive values on $\B$ is called the degree map of the table algebra, it corresponds to the augmentation map for group rings, the valency map for association schemes, and the Perron-Frobenius eigenvalue for fusion rings.  If we re-scale the basis so that $\lambda_{ii^*0}=k_i$ for all $i$, we call the new basis $\B$ {\it standard}, and if we re-scale the basis so that $\lambda_{ii^*0}=1$ for all $i$, we call the new basis $\B$ {\it transitional}. (This terminology comes from Blau's survey, see \cite{Blau09}.)  In our group algebra and adjacency algebra situations, the defining basis is standard, and for fusion rings the defining basis is transitional.  Note that re-scaling can affect the integrality of structure constants, which is necessary for $\Z \B$ to be an order in $\Q \B$.     

When $\B$ is the standard basis of a table algebra, $n = \sum_{i=0}^d \delta(b_i)$ is called the order of the table algebra, and the linear map extending $\rho(b_i) = \begin{cases} n & i=0 \\ 0 & i \ne 0 \end{cases}$ is called the standard feasible trace of $\C \B$.  The standard feasible trace induces a bilinear form $\langle a,b \rangle = \rho(ab^*)$, $a,b \in \C \B$, which leads to a formula for the centrally primitive idempotents of the table algebra (see \cite[\S5]{Higman1975}, \cite[\S 5]{Higman1987}, \cite[Proposition (9.17)]{CR1}, and \cite[Theorem 3.6]{AFM1999}).  For the standard basis of a table algebra, this formula says the centrally primitive idempotent of $\C \B$ associated to an irreducible character $\chi$ of $\C \B$ is  
\begin{equation}\label{standard-character-formulae}
e_{\chi} = \frac{m_{\chi}}{n} \sum_{i=0}^d \frac{\chi(b_i^*)}{\delta(b_i)} b_i 
\end{equation}
where the numbers $m_{\chi}$ are the positive real numbers that occur in the expression of the standard feasible trace as a linear combination of the irreducible characters of $\C \B$; i.e. $\rho = \sum_{\chi \in Irr(\C\B)} m_{\chi} \chi$.  The positive real number $m_{\chi}$ is called the {\it multiplicity} of $\chi$, and it follows from uniqueness of the Perron-Frobenius eigenvector that the degree character $\delta$ has multiplicity $m_{\delta}=1$.  
 When $\B$ is a transitional basis, the re-scalings $b_i \mapsto \delta(b_i)b_i:= B_i$ and $B_i \mapsto \frac{1}{\sqrt{\delta(B_i)}}B_i:=b_i$ allow us to change back-and-forth from a standard basis, so in the case when $\Z \B$ is a fusion ring the above formula translates to 
\begin{equation}\label{transitional-character-formulae}
e_{\chi} = \frac{m_{\chi}}{n} \sum_{i=0}^d \chi(b_i^*)b_i.
\end{equation}

\section{Zeta functions of Orders}

Let $\Lambda$ be an order in a commutative semisimple $\Q$-algebra $A$.  In this case $A$ has a unique maximal order 
$$ \Lambda_0 = \oplus_{\chi} \Z(\chi) \tilde{e}_{\chi}, $$ 
where $\chi$ runs over a set of representatives of the Galois conjugacy classes of irreducible characters of $\C \B$, and the primitive idempotents $\tilde{e}_{\chi}$ of $\Q \B$ are the sums of the primitive idempotents $e_{\psi}$ of $\C \B$ where $\psi$ runs over the Galois congujates of $\chi$, and $\Z(\chi)$ are the rings of integers of the number fields $\Q(\chi)$. 
The Solomon zeta function of $\Lambda_0$ will be the product of the Dedekind zeta functions for the number fields $\Q(\chi)$.  Recall that when $R$ is the ring of integers of a number field, its Dedekind zeta function is 
$$\zeta_R(s) = \prod_{\mathcal{P}} (1 - [R:\mathcal{P}]^{-s})^{-1}, $$ 
where $\mathcal{P}$ runs over the nonzero prime ideals of $R$. 

In most of the cases  we 
consider in this paper, our integral table algebra $\Z \B$ 
is not equal to the maximal order of $\Q \B$. In \cite{Solomon77}, Solomon proved the Euler product identity $\zeta_{\Z \B}(s) = \prod_p \zeta_{\Z_p \B}(s)$, where $p$ runs over the rational primes, and $\Z_p \B$ denotes the $p$-adic completion $\Z_p \otimes \Z \B$.  Furthermore, he showed each factor $\zeta_{\Z_p\B}(s)$ of this Euler product is equal to $\zeta_{\Lambda_{0,p}}(s)$ times the value of a polynomial $\delta_{p}(x)$ evaluated at $p^{-s}$, where $\Lambda_{0,p}$ is a maximal order of $\Q_p\B$ containing $\Z_p\B$.   In all but finitely many cases, Solomon showed $\Z_p\B$ is equal to $\Lambda_{0,p}$, the exceptions being the primes $p$ that divide either the discriminant of $\Z \B$ or the least positive integer $f$ with $f \Lambda_0 \subset \Z\B$.  

This reduces the calculation of $\zeta_{\Z \B}(s)$ to the calculation of the zeta functions of the local orders $\Z_p \B$ for finitely many primes $p$.  To calculate these, we will use the approach Bushnell and Reiner introduced in \cite{Bushnell-Reiner1980} for calculating local zeta integrals.  First, Bushnell and Reiner showed 
that when $\Lambda_p$ is a $\Z_p$-order in a maximal order $\Lambda_{0,p}$ 
of a $\Q_p$-algebra $A_p$, 
$$\zeta_{\Lambda_p}(s) =  \sum_{M \in \mathscr{G}} Z_{\Lambda_p}(\Lambda_p, M ; s), $$ 
where $\mathscr{G}$ is a set of representatives up to isomorphism for all full $\Lambda_p$-lattices in $A_p$.  For each $M \in \mathscr{G}$ the \textit{genus zeta function} $Z_{\Lambda_p}(\Lambda_p,M;s)$ computes the contribution from all $\Lambda_p$-sublattices isomorphic to $M$. In the special case of the regular module of $\Lambda_p$, Bushnell and Reiner's integral formula for the gunus zeta function reduces to 
\begin{equation}
\label{genuszetaeqn}
 Z_{\Lambda_p}(\Lambda_p, M; s)=\mu(\Aut M)^{-1} (\Lambda_p:M)^{-s} \int_{A_p^{\times}}  \Phi_{\{M:\Lambda_p\}}(x) \left\Vert x \right\Vert^s d^\times x,
\end{equation}
where   the  
index is $\displaystyle (\Lambda_p:M)=\frac{(\Lambda_p:(\Lambda_p \cap N))}{(M:(\Lambda_p \cap M))}$,  $ \Phi_{\{M:\Lambda_p \}}$ is the characteristic function in $A$ of the lattice 
\[\{M:\Lambda_p \}=\{ x \in A \, | \, Mx \subseteq \Lambda_p \},\]
the norm is $\left\Vert x \right\Vert=(Nx:N)$ where $N$ is any full $\Z_p$-lattice in $A_p$, and $d^\times x$ denotes a multiplicative Haar measure on $A_p^\times$, which is normalized so that integrating over $\Lambda_{0,p}^{\times}$ gives $1$.   The first factor in their integral formula is the inverse of $\mu(\Aut M)=\mu(\{ M: M\}^\times)$ in this same measure.  When the $\Lambda_p$-lattice $M$ is a $\Z_p$-order, $\mu(\Aut M)=\mu (M^\times)$.  

The first challenge in applying Bushnell and Reiner's integral formula is to determine a full set of representatives for the full $\Z_p\B$-lattices in $A_p$.   In \cite{BH2022}, all of the cases under consideration were orders in split semisimple $\Q_p$-algebras, which made this task a bit easier.  The next Lemma generalizes \cite[Lemma 3.2 and 3.4]{BH2022}, it will enable us to do this in the non-split situation.   

\begin{lem} Let $A$ be a commutative semisimple $\Q_p$-algebra.  Let $\mathscr{E} = \{ e_0, \dots, e_d \}$ be the set of primitive idempotents of $A$, and suppose $A \simeq K_0e_0 \oplus \dots \oplus K_d e_d$ for $p$-adic number fields $K_0, \dots, K_d$.  For each $i=0,\dots,d$ let $R_i$ be the ring of integers of $K_i$ and let $\pi_i R_i$ be the maximal ideal of $R_i$.  

Let $\Lambda$ be a full $\Z_p$-order in the maximal order $\Lambda_0 = R_0 e_0 \oplus \dots \oplus R_d e_d$ of $A$.   

Let $N$ be any $\Lambda$-lattice in $A$. Then the following hold. 

(i) $N$ has a block upper triangular basis as a $\Z_p$-module; i.e. there exists a $\Z_p$-basis 
$$ \{ v_{0,1}, v_{0,2}, \dots, v_{0,r_0}, v_{1,1}, \dots, v_{d,r_d} \} $$
of $\Lambda_0$ and powers $a_{i,j_i} \ge 0$ for $i=0,\dots,d$, $j_i=1,\dots,r_i$ such that 
$\{v_{i,1}e_i, \dots, v_{i,r_i}e_i \}$ is a $\Z_p$-basis of $R_i e_i$ for $i=0,\dots,d$; $\{ \pi_i^{a_{i,j_i}} v_{i,j_i} : i=0,\dots,d, j_i=1,\dots,r_i \}$ is a $\Z_p$-basis of $N$; and $v_{i,j_i}e_k = 0$ if $k < i$, for  $j_i = 1,\dots,r_i$, $i=1,\dots,d$, and $k=0,\dots,d-1$.

(ii) There exists a $\Lambda$-lattice $M$ in $A$ such that $M \simeq N$ as $\Lambda$-lattices, $1 \in M$, and $M \subseteq \Lambda_0$.  

(iii) Suppose $M$ and $N$ are an isomorphic pair of $\Lambda$-lattices in $\Lambda_0$ that both contain $1$.  Then any $\Lambda$-lattice isomorphism from $M$ to $N$ is realized as multiplication by a unit $u \in N \cap \Lambda_0^{\times}$ for which $u^{-1} \in M \cap \Lambda_0^{\times}$. 
\end{lem}

\begin{proof} (i).   Let $U$ be a $\Z_p$-basis of $N$.  We choose a
subset $W_0 = \{w_{0,1}, \dots, w_{0,r_0}\}$ of $U$ so that $\{w_{0,1}e_0, \dots, w_{0,r_0}e_0 \}$ is a $\Z_p$-basis of $N e_0$.  Since $R_0 e_0$ is a local ring with maximal ideal generated by $\pi_0$, we can find powers $\pi_0^{a_{0,j_0}}$ with $a_{0,j_0} \in \Z$ for $j_0=1,\dots,r_0$ such that $\pi_0^{-a_{0,j_0}}w_{0,j_0}e_0$ is a unit of $R_0 e_0$ for $j_0 = 1,\dots,r_0$.  Define $v_{0,j_0} \in \Lambda_0$ to be $\pi_0^{-a_{0,j_0}}w_{0,j_0}e_0 + w_{0,j_0}(1-e_0)$ for $j_0 = 1, \dots, r_0$.  It is easy to see that $V_0=\{v_{0,j_0}e_0 : j_0 = 1,\dots,r_0 \}$ is a $\Z_p$-basis of $R_0e_0$, and $\{\pi_0^{a_{0,j_0}}v_{0,j_0}e_0 : j_0=1,\dots,r_0 \} = \{ w_{0,j_0}e_0 : j_0=1,\dots,r_0 \}$ is a $\Z_p$-basis of $Ne_0$.  

For $u \in U - W_0$, we have that $ue_0 = w_ue_0$, where $w_u$ is a specific $\Z_p$-linear combination of $\{w_{0,j_0}: j_0=1,\dots,r_0 \}$.  Let $U_1 = \{ u - w_u : u \in U - W_0 \}$.  Then $W_0 \cup U_1$ is a $\Z_p$-basis of $N$ for which $U_1e_0 = 0$.  Repeat this process with the idempotents $e_1, \dots, e_d$ in succession, to produce $\Z_p$-independent sets $W_0, W_1, \dots, W_d$ whose union is a block upper triangular basis of $N$ with respect to the idempotents $e_0, \dots, e_d$. 

(ii).  Let $N$ be a full $\Lambda$-lattice in $A$.  Since each $K_i$ is a $p$-local field, we can multiply by an appropriate power of $p$ to arrange that $N \subseteq \Lambda_0$ but $ p^{-1}N \not\subseteq  \Lambda_0$.   If $N$ contains a unit $u$ of $\Lambda_0$, then $Nu^{-1}$ would be a $\Lambda$-lattice satisfying $1 \in Nu^{-1} \subseteq \Lambda_0$.  So suppose $N$ does not contain a unit.  Then there is a primitive idempotent $e_i$ for which $Ne_i \subseteq \pi_i R_i e_i$.  Suppose $a_i \ge 1$ is the largest power of $\pi_i$ such that $Ne_i \subsetneq \pi_i^{a_i}R_i e_i$.  Let $\phi_i : N \rightarrow \Lambda_0$ be the map $\phi_i(n) = n (1 - e_i) + \pi_i^{-a_i}ne_i$ for all $n \in N$.  It is easy to see that $\phi_i$ is $\Lambda_0$-linear and injective, so $\phi_i(N)$ is isomorphic to $N$ as a $\Lambda$-lattice and the number of primitive idempotents $e_j$ for which $\phi_i(N)e_j \subseteq \pi_j R_j e_j$ is one fewer than the number for $N$.  By repeatedly applying the maps $\phi_i$ we will eventually arrive at a $\Lambda$-lattice $M$ that is isomorphic to $N$ and will contain an element $u$ for which $ue_i \in R_ie_i - \pi_i R_i e_i$ for $i=0,\dots,d$.  So this element $u$ of $M$ is a unit of $\Lambda_0$, and (ii) follows.

(iii).  Let $k$ be a positive integer for which $p^kM \subseteq \Lambda$, and let $\phi:M \rightarrow N$ be a $\Lambda$-lattice isomorphism.  Then for all $m \in M$
$$ \phi(m) = p^{-k} p^k \phi(m) = p^{-k} \phi(p^km) = p^{-k}(p^km) \phi(1) = m \phi(1). $$
Let $u = \phi(1)$.  We have that $u \in \Lambda_0$ because $\phi(M)=N \subseteq \Lambda_0$.  Since $1 \in N$, there is a $v \in M$ for which $1 = \phi(v) = vu$. Therefore, $u \in N \cap \Lambda_0^{\times}$ and $v = u^{-1} \in M \cap \Lambda_0^{\times}$. 
\end{proof} 

Once we have a set of representatives $\{M\}$ of the genera of $\Z_p\B$-lattices in $A$, the next challenge is determining the genus zeta functions $Z_{\Z_p\B}(\Z_p\B, M; s)$. The following lemma is helpful in determining these integrals.

\begin{lem}
\label{smallintegrals}
Let $K$ be a $p$-adic number field with ring of integers $R$, whose maximal ideal is $\pi R$, such that the residue class field  $R/\pi R$ is isomorphic to  the finite field with $q$ elements $\mathbb{F}_q$, where $q=p^f$.  Let $d^\times x$ denote a multiplicative Haar measure on $K^\times$ such that $\int_{R} d^\times x = 1$ and $\left\Vert x \right\Vert=(R:Rx)^{-1}$.  Then the following hold:
\begin{enumerate}
\item For any measurable subset $S \subseteq R$ and $i \ge 0$, $\int_{\pi^i S} \left\Vert x \right\Vert^s d^\times x = q^{-is} \int_{S} \left\Vert x \right\Vert^s d^\times x$;
\item $\int_{\pi^i R-\{0\}} \left\Vert x \right\Vert^s d^\times x = \frac{q^{-is}}{1-q^{-s}}$;
\item For $i\ge 1$, denote the $i$-th higher unit group by $U^{(i)} = 1 +\pi^i R$. Then $\int_{U^{(i)}} \left\Vert x \right\Vert^s d^\times x = \frac{1}{q^{i-1}(q-1)}$. Moreover, for any $u \in (R/\pi R)^\times$, $a_k \in R/\pi R $ and $i \ge 1$, consider the subset $S=u+\sum_{k=1}^{i} a_k \pi^k + \pi^i R$. Then $\int_{S} \left\Vert x \right\Vert^s d^\times x = \frac{1}{q^{i-1}(q-1)}$.
\end{enumerate}
\end{lem}

\begin{proof}
   \begin{enumerate}
    \item Write $S$ as a disjoint union $S= \dot\bigcup_{j} \pi^j S_j$, where each $S_j \subseteq R^\times$ and $j \ge 0$. Since $\pi^i(\pi^j S_j)$ and $\pi^jS_j$ are multiplicative translates, $\int_{\pi^i(\pi^j S_j)} d^\times x = \int_{\pi^jS_j}d^\times x$. Moreover, $( R : R \pi^j)=q^{j}$, so  
    \begin{align*} \int_{\pi^i S} \left\Vert x \right\Vert^s d^\times x& = \sum_j \int_{\pi^{i+j} S_j}   \left\Vert x \right\Vert^s d^\times x  = \sum_j \int_{\pi^{i+j} S_j} q^{-(i+j)s} d^\times x =   q^{-is} \sum_j\int_{\pi^i(\pi^{j} S_j)} q^{-js}d^\times x \\
    & = q^{-is} \sum_j\int_{\pi^{j} S_j} ( R : R \pi^j)^{-s} d^\times x = q^{-is}\int_S \left\Vert x \right\Vert^s d^\times x.
    \end{align*}
   
       \item The claim follows from (1), since we have the disjoint union $R-\{0\}=\dot\bigcup_{j=0}^\infty \pi^j R^\times$ and 
    $$\int_{R-\{0\}} \left\Vert x \right\Vert^s d^\times x =  \sum_{j=0}^\infty \int_{\pi^j R^\times} \left\Vert x \right\Vert^s  d^\times x =  \sum_{j=0}^\infty q^{-js}=\frac{1}{1-q^{-s}}.$$
   
    \item Since $(R:Rx) = 1$ for any $x \in U^{(1)}$, this follows from $\int_{U^{(i)}} \left\Vert x \right\Vert^s d^\times x = \mu(U^{(i)})=\frac{1}{q^{i-1}(q-1)}$ by \cite[Proposition 3.10]{neukirch}. The second claim is due to $S=(u+\sum_{k=1}^{i} a_k \pi^k)U^{(i)}$ with $u+\sum_{k=1}^{i} a_k \pi^k \in R^\times$.
   \end{enumerate} 
   \end{proof}

 We are now ready to investigate zeta functions for integral table algebras with irrational character table.  The smallest standard integral table algebra examples have rank $3$, and occur in both the asymmetric and symmetric situations.

\begin{exa}\label{asymrk3} {\rm The standard basis $\B = \{1, b, b^*\}$ of the adjacency algebra of an asymmetric  table algebra of rank $3$ satisfies $bb^* = (2u+1)1 + ub + ub^*$ and $b^2=ub+(u+1)b^*$ for some $u \ge 0$.  (These table algebras are realized as the adjacency algebras of association schemes corresponding to doubly regular tournaments of order $n=4u+3$.) The minimal polynomial of both $b$ and $b^*$ is $\mu(x)=(x-(2u+1))(x^2+x+(u+1))$, and it follows that $\Q \B = \Q \oplus \Q\left[\frac{-1+\sqrt{-n}}{2}\right]$ for $n=4u+3$.  
Since the second component does not split over $\Q$, the idempotents of $\Q \B$ will be 
$$ e_0 = \frac{1}{n}(1 + b + b^*) \mbox{ and } e_1 = \frac{n-1}{n} 1 - \frac{1}{n}(b+b^*). $$ 

A $\Z$-basis for the maximal order $\Lambda_0$ of $\Q \B$ is $\{e_0, e_1, b e_1\}$, and when we write $1$, $b$, and $b^*$ in terms of this basis we get 
$$ 1 = e_0 + e_1, b = (2u+1)e_0 + be_1 \mbox{ and } b^* = (2u+1)e_0 - e_1 - b e_1. $$
Reducing the basis $\{1, b, b^*\}$ of $\Z \B$ gives
$$\Z \B = \langle ne_0, e_0+e_1,  (2u+1)e_0+ b e_1 \rangle.$$  
(Here we denote $\langle v_1, v_2, v_3 \rangle$ for the $\Z$-lattice spanned by $\{v_1,v_2,v_3\}$, we will also use the same notation later for locally integral $\Z_p$-lattices.) 
This tells us $[\Lambda_0: \Z \B] = n$, and so to compute the zeta function of $\Z \B$ we need to consider its completions at the primes $p$ dividing $n$.   
}\end{exa} 

\begin{exa} \label{symrk3} {\rm Let $\B = \{ 1, b_1, b_2 \}$ be the standard basis of a symmetric table algebra of rank $3$ that has an irrational character table.  In this situation there exists a $u \ge 0$ such that $b_1^2 = (2u)1 + (u-1)b_1 + ub_2$ and $b_1b_2 = ub_1+ub_2 = b_2b_1$, and the minimal polynomial of $b_1$ and $b_2$ has the irreducible factorization $\mu(x) = (x-2u)(x^2+x-u)$.  (These table algebras are realized by the adjacency algebras of association schemes corresponding to conference graphs of order $n=4u+1$ when $n$ is not a perfect square.  Payley graphs are one such example.)  For these table algebras, $\Q \B \simeq \Q \oplus \Q[\frac{-1+\sqrt{n}}{2}]$ and the maximal order $\Lambda_0$ of $\Z \B$ is $\langle e_0 , e_1, b_1e_1 \rangle$. 
Writing the elements of $\B$ in terms of this basis, we have 
$$ 1 = e_0 + e_1, b_1 = (2u)e_0 + b_1e_1, \mbox{ and } b_2 = (2u)e_0 -e_1 -b_1e_1. $$ 
Reducing the basis $\{1, b_1, b_2\}$  gives  $\Z \B = \langle e_1+e_0, b_1e_1+ 2ue_0 , ne_0 \rangle$, so $[\Lambda_0: \Z \B] = n$.  Again for the zeta function calculation this tells us we need to consider completions at the primes $p$ that divide $n$.}
\end{exa}

 In both examples, we are interested in the $p$-local zeta function of $\Z_p \B$, for odd primes $p$ dividing $n$. Let $n=p^{v_p(n)}v$, where $\gcd(p,v)=1$. Then we have three cases to consider: $v_p(n)$ is odd,  $v_p(n)$ is even and $v$ is a nonsquare modulo $p$, and $v_p(n)$ is even and $v$ is a square modulo $p$. Since we focus on table algebras of rank $3$ with irrational-valued characters, there must be some odd prime $p$ dividing $n$ such that $v_p(n)$ is odd, so we start with this case, where  one of the simple components of $\Q_p \B$ will be a ramified quadratic extension.

\begin{thm} 
\label{rank3}
Let $\B$ be the standard basis of a rank $3$ table algebra with irrational character table; i.e. one of the above two cases depending on whether the order $n \equiv \pm 1 \mod 4$.

Let $p$ be a prime such that $p$ divides $n$ but $p^2$ does not divide $n$.  

Then $\zeta_{\Z_p \B}(s) = (p^{1-2s}-p^{-s}+1)\zeta_{\Z_p}(s)^2$.
\end{thm}

\begin{proof}

We have that $p$ is an odd prime with $n=pv$ for an integer $v$ with $(p,v)=1$.
Then the completion at $p$ gives $A := \Q_p\B \cong \Q_p \oplus K$ where $K=\Q_p(\pi)$ is a ramified quadratic extension of $\Q_p$ with ring of integers $R$ and  uniformizer $\pi$ satisfying $(\pi R)^2=pR$, so the residue class field $R/\pi R \cong \mathbb{F}_p$.  Since $\frac12 \in \Z_p$, we have that $R = \Z_p[\sqrt{\pm n}]$, the sign depending on the congruence of $n$ mod $4$.  The uniformizer of $R$ can be taken to be $\pi=\sqrt{\pm n}$ since $\pi^2 = \pm pv$ will always have $p$-valuation $1$ in $\Z_p$. 

The maximal order of $\Q_p \B$ is $\Lambda'_p \cong \Z_p \oplus R$, where $R=\Z_p[\pi]$. Let $e_0$ and $e_1$ be the primitive idempotents of $\Q_p\B$, with $\Q_p\B = \Q_pe_0 \oplus K e_1$.  In $\Q_p \B e_1$, the element $e_1+2b_1e_1$ can be identified with $\pi e_1 = \sqrt{\pm n} e_1$.  Then $\Lambda'_p$ has $\Z_p$-basis $\{e_0,e_1,\pi e_1 \}$, and with respect to this basis, $\Lambda_p \coloneqq \Z_p\B = \langle e_0+e_1, n e_0, \pi e_1 \rangle$. Therefore, we can represent $\Lambda_p$ as   
$$\Lambda_p=\{(\alpha, \beta) \in \Z_p \oplus R : \alpha \pmod{p} \equiv \beta \pmod{\pi} \}.$$

From this point the calculation of the zeta function $\zeta_{\Lambda_p}$ follows an identical argument as the calculation of the zeta function of the group ring for the cyclic group of order $p$ outlined in \cite[Section 3.4]{Bushnell-Reiner1980}. 
Since $(\Lambda_p':\Lambda_p)=p$, there are only two genera of $\Lambda_p$-lattices, with representatives $\Lambda'_p$ and $\Lambda_p$.  Therefore, we need two genus zeta functions $Z(\Lambda_p',\Lambda_p; s)$ and $Z(\Lambda_p,\Lambda_p; s)$. 

We first calculate $$Z(\Lambda_p',\Lambda_p; s)= \mu({\Lambda'_p}^{\times})^{-1} (\Lambda_p:\Lambda'_p)^{-s} \int_{A^\times}  \Phi_{\{\Lambda'_p:\Lambda_p\}}(x) \left\Vert x \right\Vert_A^s d^\times x.$$ Note that $\{\Lambda_p':\Lambda_p\}=p\Z_p \oplus \pi R$, so by Lemma \ref{smallintegrals},  
 $$Z(\Lambda_p',\Lambda_p; s)= p^s \int_{p\Z_p-\{0\}} \left\Vert a \right\Vert_{\Q_p}^s d^\times a    \int_{\pi R -\{0\}}  \left\Vert b\right\Vert_K^s  d^\times b = p^{-s}(1-p^{-s})^{-2}.$$ 

On the other hand, $\{\Lambda_p:\Lambda_p\}=\Lambda_p$, and  $A^\times \cap \Lambda_p$ equals the disjoint union between $ \Lambda_p^\times \dot\cup (A^\times \cap (p \Z_p \oplus \pi R))$, the latter being the domain of integration in the previous genus zeta function. Therefore, 
$$\begin{array}{rcl} 
Z(\Lambda_p,\Lambda_p; s) &=& \mu(\Lambda_p^\times)^{-1} (\Lambda_p:\Lambda_p)^{-s} \left(\int_{\Lambda_p^\times} \left\Vert x \right\Vert_A^s d^\times x + p^{-2s}(1-p^{-s})^{-2}\right) \\
&=&1+(p-1)p^{-2s}(1-p^{-s})^{-2}.
\end{array}$$

Since the local Dedekind zeta function is $\zeta_{\Z_p}(s)=(1-p^{-s})^{-1}$, we get  $$\zeta_{\Lambda_p}(s)=(p^{1-2s}-p^{-s}+1)\zeta_{\Z_p}(s)^2.$$ 
 \end{proof}

It is interesting to see that $\zeta_{\Lambda_p}(s)=\zeta_{\Z_pC_p}(s)$, the (local) zeta function of the group ring for the cyclic $p$-group $C_p$.  This was also the case for general association schemes of order $p$ in \cite{Hanaki-Hirasaka2016}.

When we consider the extension of the previous result to the case where $n$ is divisible by a higher powers of $p$, we need to first find representatives for the isomorphism classes of $\Z_p\B$-lattices between $\Z_p\B$ and the maximal order of $\Q_p\B$.  

\begin{lem} 
\label{lattice-lemma}
Let $\B$ be the defining basis of a rank $3$ table algebra of order $n$ with irrational character table.  From Examples \ref{asymrk3} and \ref{symrk3}, we know that $n>1$ is odd, $n$ is not a square, and its character table is determined by the congruence of $n$ mod $4$.   

Let $p$ be an odd prime for which  $v_p(n)=2m+1$ is odd. Let $v$ be the integer coprime to $p$ for which $\pm n = p^{2m+1}v$ and $n \equiv \pm 1 \mod 4$.

\begin{enumerate} 
\item $\Q_p\B \simeq \Q_p \oplus \Q_p(\sqrt{\pm n})$, and the maximal order $\Lambda_0$ of $\Q_p\B$ is isomorphic to $\Z_p \oplus \Z_p[\pi]$, where $\pi^2 = pv$. 

\item Let $\Lambda=\Z_p\B$ and let $\Lambda_0$ be the maximal $\Z_p$-order of $\Q_p\B$. The $\Lambda$-lattices satisfying $\Lambda \subseteq M \Lambda_0$ are $\Z_p$-lattices $M(r,i,j) = \langle e_0+e_1, re_0+p^i \pi, p^j e_0 \rangle$, where $0 \le i \le m$, $0 \le j \le 2m+1$, and $0 \le r < p^{2m+1}$ satisfy the conditions $m+i+1 \ge j$ and $m-i \ge j-k$. 

\item  Two of the $\Lambda$-lattices $M(r,i,j)$ and $M(s,i',j')$ with $(r,i,j) \ne (s,i',j'
)$
in the previous list will be isomorphic as $\Lambda$-lattices if and only if $i=i'$, $j=j'$, and one of the following holds: 

$\bullet$ $r,s \not\equiv 0 \mod p$, 

$\bullet$ $s=0$ and $1 \le 2i+1 \le v_p(r) < j$, or 

$\bullet$ $1 \le v_p(s)=2i+1 \le v_p(r) < j$.
\end{enumerate} 

\end{lem} 

\begin{proof} 
\begin{enumerate} 
\item Since $p$ is odd and $\Q \B \simeq \Q \oplus \Q(\frac{-1 + \sqrt{\pm n}}{2})$, when we complete the second component with respect to the $p$-adic valuation its ring of integers will be equal to $\Z_p[\sqrt{\pm n}]$, the sign agreeing with $n \equiv \pm 1 \mod 4$.  Since $v_p(n)=2m+1$, 
$\Q_p(\sqrt{\pm n})$ will be a ramified quadratic extension of $\Q_p$.  Writing $\pm n = p^{2m+1}v$ with $v$ a positive or negative integer coprime to $p$, we have that the maximal order $\Lambda_0$ in $\Q_p \B$ is $\Z_p \oplus \Z_p[\pi]$, where $\pi^2 = pv$.  

\item Let $e_0$ and $e_1$ be the primitive idempotents of $\Lambda_0$.  We fix the ordered  $\Z_p$-basis $\{\pi e_1, e_1, e_0\}$ of $\Lambda_0$. Then $\Lambda$ is the $\Z_p$-lattice $\langle  p^m\pi e_1, e_1+e_0, p^{2m+1}e_1 \rangle$. This is in Hermite normal form with respect to our fixed basis of $\Lambda_0$, which we represent with the matrix  $ \left(\begin{smallmatrix} p^m& 0 & 0 \\ 0 & 1 & 1 \\ 0 & 0 & p^{2m+1} \end{smallmatrix}\right).$

The $\Lambda$-lattices $M$ such that $\Lambda \subseteq M \subseteq \Lambda_0$ will correspond to the matrix forms given by $\left(\begin{smallmatrix} p^i& 0 & r \\ 0 & 1 & 1 \\ 0 & 0 & p^{j}\end{smallmatrix}\right),$ where $0 \le i \le m$ and $ 0 \le j \le 2m+1$, and $0 \le r < p^{2m+1}$, that satisfy the arithmetic conditions imposed by the requirements $\Lambda \subseteq M$ and $\Lambda M \subseteq M$.  

Let $M(r,i,j)$ be the $\Z_p$-lattice whose Hermite normal form is determined by the parameters $r$, $i$, and $j$.  If $r=0$, it is clear that $\Lambda \subseteq M(0,i,j)$.  Suppose $r \ne 0$, and write $r = p^k r_0$ with $(r_0,p)=1$ and $0 \le k < j$.  It is easy to see that $\Lambda \subseteq M(r,i,j)$ will hold if and only if $p^m\pi e_1 \in M$, in which case
$$\begin{array}{rcl}
p^m \pi e_1 &=& p^{m-i}(p^i \pi e_1 + r_0p^k e_0)-p^{m-i+k-j}r_0(p^j e_0),
\end{array}$$ 
 for $p^{m-i+k-j}r_0 \in \Z_p$, so it must be the case that $m+k \ge i+j$.  

Suppose that either $r=0$, or $r = p^kr_0$ with $(r_0,p)=1$, $0 \le k < j$, and $m+k \ge i+j$.  To show $\Lambda M(r,i,j) \subseteq M(r,i,j)$, it suffices to show $M(r,i,j)$ is closed under multiplication by the Hermite normal form basis elements of $\Lambda$.  Since it is clear that $\Lambda$ contains $p^{2m+1}\Lambda_0$ and $\Lambda \subseteq M(r,i,j)$, we only need to check that $M(r,i,j)$ is closed under multiplication by $p^m\pi e_1$, and this comes down to $p^m\pi e_1 (re_0 + p^i \pi e_1) = p^{m+i+1}ve_1 \in M(r,i,j)$.  So there needs to be a $\Z_p$-solution to 
$$ p^{m+i+1}v e_1 = a(e_0+e_1)+b(p^kr_0e_0+p^i \pi e_1)+c(p^je_0) = (a+bp^kr_0+cp^j)e_0+(a+bp^i\pi)e_1. $$
Solving this we get $a=p^{m+i+1}v$, $b=0$, and $c=p^{m+i+1-j}v$. This is a $\Z_p$-solution if and only if $m+i+1 \ge j$. Note that this condition does not depend on whether or not $r=0$.  So $M(r,i,j)$ is a $\Lambda$-lattice with $\Lambda \subseteq M(r,i,j) \subseteq \Lambda_0$ when $m+i+1\ge j$, or when $r\ne 0$, $v_p(r)=k$, and $m+k \ge i+j$.  

\item Now we consider the $\Lambda$-lattice isomorphism classes of the $M(r,i,j)$ satisfying the above conditions using condition (iii) of Lemma 3.1. In order for two of these lattices $M = M(s,i',j')$ and $N=M(r,i,j)$ 
to be isomorphic, there must be a unit $u \in N \cap \Lambda_0$ for which $u^{-1} \in M \cap \Lambda_0$ and $uM=N$.  Since multiplication by a unit of $\Lambda_0$ preserves the index of $M$ in $\Lambda_0$, we must have that $M$ and $N$ share the same $i$ and $j$ parameters.  Since $j=0$ forces $r=s=0$, when $M \ne N$ we must have that $j>0$. For $u = \alpha(e_0+e_1)+\beta(re_0+p^i \pi e_1) + \gamma(p^je_0) \in N$ to be a unit of $\Lambda_0$, we must have that $\alpha \not\equiv 0 \mod p$ and $(\alpha +\beta r + \gamma p^j) \not\equiv 0 \mod p$.  Since scaling by units of $\Z_p$ preserves $M$ and $N$, we can assume $\alpha=1$, and our condition for $u$ to be a unit becomes $(1 + \beta r) \not\equiv 0 \mod p$.  

Next, we determine the conditions on $\beta$ and $\gamma$ in order for $u^{-1} \in M=M(s,i,j)$.  We have that $u^{-1} = (1+\beta r + \gamma p^j)^{-1} e_0 + (1 + \beta p^i \pi)^{-1} e_1$, so the inverse of its second component is $(1 + \beta p^i \pi)^{-1} = 1 -\beta p^i \pi + \sum_{\ell > 1} (-1)^{\ell} \beta^{\ell} (p^i \pi)^{\ell}$. Therefore, for $u^{-1} \in M$ we must be able to find $\delta, \epsilon, \phi \in \Z_p$ for which 
$$ u^{-1} = \delta(e_0+e_1)+\epsilon(se_0+p^i\pi e_1) + \phi(p^je_0) = (\delta + \epsilon s + \phi p^j)e_0 + (\delta + \epsilon p^i \pi) e_1. $$
Now, 
$$\begin{array}{rcl}
(\delta + \epsilon p^i\pi) &=& (1 + \beta p^i \pi)^{-1} \\
&=& 1 - \beta p^i \pi + \beta^2 p^{2i} \pi^2 - \beta^3 p^{3i} \pi^3 + \dots \\
&=& (1 + \beta^2 p^{2i}(pv) + \beta^4 p^{4i}(pv)^2 + \dots) - \beta p^i \pi (1 + \beta^2 p^{2i}(pv) + \beta^4 p^{4i}(pv)^2 + \dots),  
\end{array}$$ 
so $(1 + \beta p^i \pi)^{-1} = \delta - \beta \delta p^i \pi$ with $\delta = \sum_{\ell \ge 0} \beta^{2 \ell} p^{2i \ell} (pv)^{\ell} = (1 - \beta p^{2i+1}v)^{-1}  \in \Z_p$.  We also have $\epsilon = - \beta \delta$.  We must also have $(\delta - \beta \delta s + \phi p^j) = (1 + \beta r + \gamma p^j)^{-1}$.  Let $\alpha = (1+\beta r)^{-1}$.  Then 
$$(1+\beta r + \gamma p^j)^{-1} = \alpha( 1 - \gamma \alpha p^j + \gamma^2 \alpha^2 p^{2j} - \gamma^3 \alpha^3 p^{3j} + \dots), $$
so we must have $(\delta - \beta \delta s) \equiv \alpha \mod p^j$, which is equivalent to $(1 - \beta^2 p^{2i+1} v) \equiv (1 - \beta s) (1 + \beta r) \mod p^j$.  This reduces to $\beta^2(rs - p^{2i+1}v) \equiv \beta(r-s) \mod p^j$.  Once we have found a $\beta \in \Z_p$ satisfying this congruence the existence of a suitable $\psi \in \Z_p$ for which $u^{-1} = (e_0+e_1)-\beta \delta(se_0+p^i\pi e_1) + \phi p^j e_0$ will be automatic. 

Finally, consider the condition $uM = N$.  Since $u(p^je_0) = (1+\beta r) p^j e_0$, we have that $uM = \langle e_0+e_1, u(se_0+p^i \pi e_1), p^j e_0 \rangle$, and so showing $uM=N$ reduces to showing $u(se_0 + p^i \pi e_1) \in N$.  Now,  
$$\begin{array}{rcl} 
u(se_0+p^i \pi) &=& (se_0+p^i \pi e_1) ((1+\beta r + \gamma p^j)e_0 + (1 + \beta p^i \pi) e_1 \\
&=& (s + \beta sr + s \gamma p^j) e_0 + (\beta p^{2i+1}v + p^i \pi ) e_1 \\ 
&=& \beta p^{2i+1}(e_0+e_1)+(re_0+p^i \pi)+(s-r+\beta rs-\beta p^{2i+1}v + s\gamma p^j) e_0, 
\end{array}$$
so this will lie in $N$ if and only if $\beta$ is a solution to $\beta(sr-p^{2i+1}v)\equiv r-s \mod p^j$.  Note that this condition implies the condition needed for $u^{-1} \in M$.

The existence of a solution to $\beta(sr-p^{2i+1}v) \equiv (r-s) \mod p^j$ depends on $j, i, r$, and $s$. If $s=0$ and $r \ne 0$ with $v_p(r)=k$, then the congruence has a solution if and only if $j >k\ge 2i+1 \ge 1$.  When $r,s \ne 0$, we can assume $j > v_p(r) =k \ge v_p(s) = \ell \ge 0$.  In the case when $2i+1 \ge j$, the congruence reduces to $\beta(rs) \equiv (r-s) \mod p^j$, and this will have a solution only in the case $k=\ell=0$.  In the case $2i+1 < j$, a solution exists if and only if $k=\ell=0$ or $1 \le \ell = 2i+1 \le k < j$.
\end{enumerate}
\end{proof}

We can now compute the $p$-local zeta function of a rank $3$ integral table algebra when the order $n$ has $v_p(n)=3$ for an odd prime $p$.

\begin{thm} 
\label{rank3highpower}
Let $\B$ be the standard basis of a rank $3$ table algebra order $n=p^3v$, $(p,v)=1$, that has an irrational character table.

Then $\zeta_{\Z_p \B}(s) =  (p^{4-8s} - p^{3-7s} + p^{3-6s}  + p^{3-5s} - p^{2-5s}  + p^{2-3s} -p^{1-3s}  + p^{1-2s}  - p^{-s} + 1)\zeta_{\Z_p}(s)^2$.
\end{thm}

\begin{proof}  Let $\Lambda=\Z_p\B$. The maximal order of $\Q_p\B$ is $\Lambda_0 \cong R\oplus \Z_p$, where $R=\Z_p[\pi]$ with residue class field $R/\pi R \cong \mathbb{F}_p$. We denote the $i$-th higher unit group in $R$ by $U_R^{(i)}$ and the $j$-th higher unit group in $\Z_p$ by $U_{\Z_p}^{(j)}$.

Since we are in the $m=1$ case of Lemma \ref{lattice-lemma}, a complete list of representatives of the isomorphism classes of $\Lambda$-lattices $M(r,i,j)$ that lie strictly between $\Lambda=\Z_p\B=M(0,1,3)$ and $\Lambda_0=M(0,0,0)$ is: $M(0,1,2)$, $M(0,1,1)$, $M(0,0,2)$, $M(0,1,0)$, $M(0,0,1)$, and $M(1,0,1)$.  The indices of $\Lambda$ in these $\Lambda$-lattices are: $1$ for $\Lambda$, $p$ for $M(0,1,2)$, $p^2$ for $M(0,1,1)$ and $M(0,0,2)$,  $p^3$ for $M(0,1,0)$, $M(0,0,1)$, and $M(1,0,1)$, and $p^4$ for $\Lambda_0$.  

To compute $\mu(\Aut M )^{-1}$ for each of the lattices above, we note all of them except $M(0,0,2)$ and $M(1,0,1)$ are orders. In these two cases, we calculate that $\{ M(0,0,2): M(0,0,2)\} = M(0,1,2)$ and $\{M(1,0,1) : M(1,0,1) \} = M(0,1,1)$. We can compute the indices of the groups of units of their automorphism groups in $\Lambda_0^{\times}$ by tracking the unit groups of their simple components $Me_0$ and $Me_1$ and congruence conditions that exist between these components. Doing so gives the following indices of the automorphism groups: 
$p^3(p-1)$ for $\Lambda$, $p^2(p-1)$ for $M(0,1,2)$ and $M(0,0,2)$, $p(p-1)$ for $M(0,1,1)$ and $M(1,0,1)$, $p$ for $M(0,1,0)$, $(p-1)$ for $M(0,0,1)$, and $1$ for $\Lambda_0$. 

For each $M(r,i,j)$ found above, we compute the lattices  $\{ M(r,i,j): \Lambda\}$ in Table \ref{table01}, 

\begin{table}[h]
\caption{Complementary lattices}
\label{table01}
\begin{tabular}{c|c}
    $M(r,i,j)$ & $\{M(r,i,j):\Lambda\}$   \\
     \hline
    $M(0,0,0)$  &   $\langle p^2 \pi e_1, p^3e_1, p^3e_0\rangle$ \\
    $M(0,1,0)$  &   $\langle p\pi e_1,p^3 e_1, p^3e_0\rangle$ \\
    $M(0,0,1)$  &   $\langle p^2 \pi e_1, p^2(e_1+e_0), p^3 e_0 \rangle$ \\
    $M(1,0,1)$  &   $\langle p\pi e_1+p^2 v e_1+p^2v e_0, p^3e_1, p^3e_0\rangle$ \\
    $M(0,1,1)$  &   $\langle p\pi e_1, p^2(e_1+e_0), p^3 e_0\rangle$ \\
    $M(0,0,2)$  &   $\langle p^2 \pi e_1, p(e_1+e_0), p^3 e_0\rangle$ \\
    $M(0,1,2)$  &   $\langle p\pi e_1, p(e_1+e_0), p^3 e_0\rangle$ \\
    $M(0,1,3)$  &   $\langle p\pi e_1, e_1+e_0, p^3 e_0\rangle = \Z_p\B$ \\
\end{tabular}
\end{table}

Next, we calculate the genus zeta function for each $\Lambda$-lattice in the table. In the first case  $Z(M(0,0,0), \Lambda; s)$, we notice that $\langle p^2 \pi e_1, p^3 e_1, p^3 e_0\rangle \cong \pi^5 R\oplus p^3 \Z_p$, so by Lemma \ref{smallintegrals}, we have
\begin{equation}
\label{rank3m000}
\int_{A^\times \cap (\pi^5 R \oplus p^3 \Z_p)} \left\Vert x \right\Vert_A^s d^\times x =  \int_{ \pi^5 R -\{0\}} \left\Vert a \right\Vert_{\Q_p[\pi]}^s d^\times a  \int_{p^3\Z_p -\{0\}} \left\Vert b \right\Vert_{\Q_p}^s d^\times b  = p^{-8s}\zeta^2,    
\end{equation}
where  $\zeta=(1-p^{-s})^{-1}$.
Therefore, 
\begin{equation}
\label{rank3z000}
Z(M(0,0,0), \Lambda; s) = p^{-4s}\zeta^2.
\end{equation}

For the remaining lattices, we examine a general element, and record in tables the possible disjoint subsets such an element could belong to, and the value of the integral corresponding to each subset as it follows from Lemma \ref{smallintegrals}. To illustrate this, consider $M(0,1,0)$. An element of $\{M(0,1,0): \Lambda\}= \langle p\pi e_1,p^3 e_1, p^3e_0\rangle \subseteq R \oplus \Z_p$ is of the form $(\pi^3(a_0+\pi^2R), p^3\Z_p)$, where $a_0 \in \mathbb{F}_p$. We have two choices: either $a_0=0$ or $a \in \mathbb{F}_p^\times$. When $a_0=0$, we get the subset $\pi^5 R  \oplus p^3 \Z_p$, whose integral we calculated in Equation \ref{rank3m000}. When $a_0=u \in \mathbb{F}_p^\times$, we obtain $(\pi^3(u+\pi^2R), p^3\Z_p) = u \pi^3 U_R^{(2)} \oplus p^3 \Z_p$, which is a multiplicative translate of $\pi^3 U_R^{(2)}\oplus p^3 \Z_p$ by $(u, 1) \in \Lambda_0^\times$, and therefore by Lemma \ref{smallintegrals}, gives the integral $\frac{p^{-6s}}{p(p-1)}\zeta$. Since there are $(p-1)$ choices for $u \in \mathbb{F}_p^\times$, the value of the integral over the disjoint union $\displaystyle \dot\bigcup_{u \in \mathbb{F}_p^\times}  u \pi^3 U_R^{(2)} \oplus p^3 \Z_p$ is $\frac{p^{-6s}}{p}\zeta$. This data is summarized in Table \ref{tablem010}.

\begin{table}[h]
\caption{Integrals for $M(0,1,0)$}
\label{tablem010}
\begin{tabular}{ |c|c|c| } 
 \hline
 $a_0$ & Subset & Value of the integral   \\ 
 \hline
 $0$ & $A^\times \cap (\pi^5 R  \oplus p^3 \Z_p)$ & Equation \ref{rank3m000} \\ 
 $u \in \mathbb{F}_p^\times$ & $(p-1)$ multiplicative translates of $\pi^3 U_R^{(2)} \oplus p^3\Z_p$ & $\frac{p^{-6s}}{p}\zeta$\\ 
 \hline
\end{tabular}
\end{table}

Adding the two integrals, we get 

\begin{align}
\label{rank3m010}
\int_{A^\times \cap \langle p\pi e_1,p^3 e_1, p^3e_0\rangle } \left\Vert x \right\Vert_A^s d^\times x 
& = p^{-8s}\zeta^2+ \frac{p^{-6s}}{p}\zeta = \frac{1}{p}(p^{1-8s} -p^{-7s}+p^{-6s})\zeta^2,  
\end{align}
and 
\begin{equation}
\label{rank3z010}
Z(M(0,1,0), \Lambda; s) = p\cdot p^{3s} \frac{1}{p}(p^{1-8s} -p^{-7s}+p^{-6s})\zeta^2= (p^{1-5s} -p^{-4s}+p^{-3s})\zeta^2 .
\end{equation}

For $M(0,0,1)$, an element of $\langle p^2 \pi e_1, p^2(e_1+e_0), p^3 e_0 \rangle $ belongs to $(\pi^4(a_0v^{-2}+\pi R), p^2(a_0+p\Z_p))$, with $a_0 \in \mathbb{F}_p$. The disjoint subsets and corresponding integrals are recorded in Table \ref{tablem001}.

\begin{table}[h]
\label{tablem001}
\caption{Integrals for $M(0,0,1)$}
\begin{tabular}{ |c|c|c| } 
 \hline
 $a_0$ & Subset & Value of the integral   \\ 
 \hline
 $0$ & $A^\times \cap (\pi^5 R  \oplus p^3 \Z_p)$ & Equation \ref{rank3m000} \\ 
 $u \in \mathbb{F}_p^\times$ & $(p-1)$ multiplicative translates of $\pi^4 U_R^{(1)} \oplus p^2 U_{\Z_p}^{(1)}$ & $\frac{p^{-6s}}{p-1}$\\ 
 \hline
\end{tabular}
\end{table}

Adding the integrals, we obtain

\begin{align}
\label{rank3m001}
\int_{A^\times \cap \langle p^2 \pi e_1, p^2(e_1+e_0), p^3 e_0 \rangle} \left\Vert x \right\Vert_A^s d^\times x & = p^{-8s}\zeta^2 + \frac{p^{-6s}}{p-1}= \frac{1}{p-1}(p^{1-8s}-2p^{-7s}+p^{-6s})\zeta^2.    
\end{align}
and 
\begin{equation}
\label{rank3z001}
Z(M(0,0,1), \Lambda; s) = p^{3s}(p-1) \frac{1}{p-1}(p^{1-8s}-2p^{-7s}+p^{-6s})\zeta^2 = (p^{1-5s}-2p^{-4s}+p^{-3s})\zeta^2.
\end{equation}

In the case of $M(1,0,1)$, an element of $\langle p\pi e_1+p^2ve_1+p^2ve_0, p^3e_1, p^3e_0\rangle$ belongs to $(\pi^3(a_0v^{-1}+a_0v^{-1}\pi + \pi^2R), p^2(a_0v+p\Z_p))$, $a_0 \in \mathbb{F}_p$. This case is recorded in Table \ref{tablem101}.

\begin{table}[h]
\caption{Integrals for $M(1,0,1)$}
\label{tablem101}
\begin{tabular}{ |c|c|c| } 
 \hline
 $a_0$ & Subset & Value of the integral  \\ 
 \hline
 $0$ & $A^\times \cap (\pi^5 R  \oplus p^3 \Z_p)$ & Equation \ref{rank3m000} \\ 
 $u \in \mathbb{F}_p^\times$ & $(p-1)$ multiplicative translates of $\pi^3 U_R^{(2)} \oplus p^2 U_{\Z_p}^{(1)}$ & $\frac{p^{-5s}}{p(p-1)}$\\ 
 \hline
\end{tabular}
\end{table}

This gives 
\begin{align*}
\int_{A^\times \cap \langle p\pi e_1+p^2ve_1+p^2ve_0, p^3e_1, p^3e_0\rangle} \left\Vert x \right\Vert_A^s d^\times x 
&= \frac{1}{p(p-1)}(p^{2-8s}-p^{1-8s}+p^{-7s}-2p^{-6s}+p^{-5s})\zeta^2,     
\end{align*}
and 
\begin{equation}
\label{rank3z101}
Z(M(1,0,1), \Lambda; s) = (p^{2-5s}-p^{1-5s}+p^{-4s}-2p^{-3s}+p^{-2s})\zeta^2.
\end{equation}.

For $M(0,1,1)$,   an element of $\langle p\pi e_1, p^2(e_1+e_0), p^3e_0\rangle$ belongs to $(\pi^3(a_0v^{-1}+a_1v^{-2}\pi+ \pi^2R), p^2(a_1+p\Z_p))$, $a_0, a_1  \in \mathbb{F}_p$. The choices for $a_0, a_1$ are recorded in Table \ref{tablem011}.

\begin{table}[h]
\caption{Integrals for $M(0,1,1)$}
\label{tablem011}
\begin{tabular}{ |c|c|c| } 
 \hline
 $(a_0, a_1)$ & Subset & Value of the integral  \\ 
 \hline
 $(0,0)$ or $(u_1, 0)$ & $A^\times \cap \langle p\pi e_1, p^3e_1, p^3 e_0\rangle$ & Equation \ref{rank3m010} \\ 
  $(0,u_2)$ & $(p-1)$ translates of  $\pi^4 U_R^{(1)} \oplus p^2 U_{\Z_p}^{(1)} $ & $\frac{p^{-6s}}{p-1}$ \\ 
   $(u_1,u_2) \in (\mathbb{F}_p^\times)^2$ & $(p-1)^2$ translates  of  $\pi^3 U_R^{(2)} \oplus p^2 U_{\Z_p}^{(1)}$ & $\frac{p^{-5s}}{p}$\\ 
 \hline
\end{tabular}
\end{table}
Adding up the integrals, we obtain
\begin{align}
\label{rank3m011}
\int_{A^\times \cap \langle p\pi e_1, p^2(e_1+e_0), p^3e_0\rangle} \left\Vert x \right\Vert_A^s d^\times x 
&= \frac{1}{p(p-1)}(p^{2-8s}-2p^{1-7s}+p^{-6s}+p^{1-5s}-p^{-5s})\zeta^2,     
\end{align}
and 
\begin{equation}
\label{rank3z011}
Z(M(0,1,1), \Lambda; s) = (p^{2-6s}-2p^{1-5s}+p^{-4s}+p^{1-3s}-p^{-3s})\zeta^2.
\end{equation}.

\medskip
For $M(0,0,2)$,   an element of $\langle p^2\pi e_1, p(e_1+e_0), p^3e_0\rangle$ belongs to $(\pi^2(a_0v^{-1}+a_1v^{-2}\pi^2+ \pi^3R), p(a_0+a_1p+p^2\Z_p))$, $a_0, a_1  \in \mathbb{F}_p$. We have the following choices, outlined in Table \ref{tablem002}.

\begin{table}[h]
\caption{Integrals for $M(0,0,2)$}
\label{tablem002}
\begin{tabular}{ |c|c|c| } 
 \hline
 $(a_0, a_1)$ & Subset & Value of the integral \\ 
 \hline
 $(0,0)$ or $(0, u_2)$ & $ A^\times \cap \langle p^2\pi e_1, p^2(e_1+e_0), p^3 e_0\rangle$ & Equation \ref{rank3m001}\\ 
  $(u_1,0)$ or $(u_1, u_2)$ & $p(p-1)$ translates of $\pi^2 U_R^{(3)}\oplus p U_{\Z_p}^{(2)}$ & $\frac{p^{-3s}}{p^2(p-1)}$ \\ 
 \hline
\end{tabular}
\end{table}
Adding them up, we obtain
\begin{align*}
\int_{A^\times \cap \langle p^2\pi e_1, p(e_1+e_0), p^3e_0\rangle} \left\Vert x \right\Vert_A^s d^\times x 
&= \frac{1}{p^2(p-1)}(p^{3-8s}  - 2 p^{2-7s}+ p^{2-6s}  + p^{-5s} - 2p^{-4s}+ p^{-3s})\zeta^2,     
\end{align*}
and 
\begin{equation}
\label{rank3z002}
Z(M(0,0,2), \Lambda; s) = (p^{3-6s}  - 2 p^{2-5s}+ p^{2-4s}  + p^{-3s} - 2p^{-2s}+ p^{-s})\zeta^2.
\end{equation}.

\medskip
For $M(0,1,2)$,   an element of $\langle p\pi e_1, p(e_1+e_0), p^3e_0\rangle$ belongs to $(\pi^2(a_0v^{-1}+a_1 v^{-1}\pi+a_2v^{-2}\pi^2+ \pi^3R), p(a_0+a_2p+p^2\Z_p))$, with $a_0, a_1, a_2  \in \mathbb{F}_p$. However, based on previously computed integrals, we only need to consider the choice of $a_0$, as described in Table \ref{tablem012}.

\begin{table}[h]
\caption{Integrals for $M(0,1,2)$}
\label{tablem012}
\begin{tabular}{ |c|c|c| } 
 \hline
 $a_0$ & Subset & Value of the integral  \\ 
 \hline
 $0$ & $A^\times \cap \langle p\pi e_1, p^2(e_1+e_0), p^3 e_0\rangle$ & Equation \ref{rank3m011}\\ 
 $u \in \mathbb{F}_p^\times$ & $p^2(p-1)$ multiplicative translates of $\pi^2 U_R^{(3)} \oplus p U_{\Z_p}^{(2)}$ & $\frac{p^{-3s}}{p(p-1)}$\\ 
 \hline
\end{tabular}
\end{table}
Adding them up, we obtain
\begin{align}
\label{rank3m012}
\int_{A^\times \cap \langle p\pi e_1, p(e_1+e_0), p^3e_0\rangle} \left\Vert x \right\Vert_A^s d^\times x 
&= \frac{1}{p(p-1)}(p^{2-8s}- 2p^{1-7s} + p^{-6s} + p^{1-5s} - 2p^{-4s} + p^{-3s})\zeta^2,     
\end{align}
and 
\begin{equation}
\label{rank3z012}
Z(M(0,1,2), \Lambda; s) = (p^{3-7s}  - 2p^{2-6s} + p^{1-5s}+ p^{2-4s}- 2p^{1-3s} + p^{1-2s})\zeta^2.
\end{equation}.

Finally, for  $M(0,1,3)$,   an element of $\Lambda = \langle p\pi e_1, e_1+e_0, p^3e_0\rangle$ belongs to $(a_0+a_1v^{-1}\pi^2+a_2v^{-1}\pi^3+ a_3v^{-2}\pi^4+\pi^5R), a_0+a_1p+a_3p^2+p^3\Z_p))$, with $a_0, a_1, a_2,a_3  \in \mathbb{F}_p$. Again, we need only consider the choice of $a_0$, recorded in Table \ref{tablem013}.

\begin{table}[h]
\caption{Integrals for $M(0,1,3)$}
\label{tablem013}
\begin{tabular}{ |c|c|c| } 
 \hline
 $a_0$ & Subset & Value of the integral  \\ 
 \hline
 $0$ & $A^\times \cap \langle p^2\pi e_1, p(e_1+e_0), p^3 e_0\rangle$ & Equation \ref{rank3m012}\\ 
 $u \in \mathbb{F}_p^\times$ & $\Lambda^\times$ & $\frac{1}{p^3(p-1)}$\\ 
 \hline
\end{tabular}
\end{table}
Adding up these integrals, we get 
\begin{align*}
\int_{A^\times \cap \Lambda} \left\Vert x \right\Vert_A^s d^\times x 
&= \frac{1}{p^3(p-1)}(p^{4-8s} - 2 p^{3-7s}  + p^{2-6s} + p^{3-5s}- 2p^{2-4s}  + p^{2-3s}  + p^{-2s} - 2p^{-s} + 1)\zeta^2 ,     
\end{align*}
and 
\begin{equation}
\label{rank3z013}
Z(M(0,1,3), \Lambda; s) = (p^{4-8s} - 2 p^{3-7s}  + p^{2-6s} + p^{3-5s}- 2p^{2-4s}  + p^{2-3s}  + p^{-2s} - 2p^{-s} + 1)\zeta^2.
\end{equation}.

Finally, adding Equations \ref{rank3z000}, \ref{rank3z010}, \ref{rank3z001}, \ref{rank3z101}, \ref{rank3z011}, \ref{rank3z002}, \ref{rank3z012}, and \ref{rank3z013}, we obtain      

\begin{equation}
\zeta(\Lambda; s) = (p^{4-8s} - p^{3-7s} + p^{3-6s}  + p^{3-5s} - p^{2-5s}  + p^{2-3s} -p^{1-3s}  + p^{1-2s}  - p^{-s} + 1)\zeta^2.
\end{equation}
\end{proof}

The two cases left to examine are for $v_p(n)$ even. If $v$ is a square modulo $p$, then the algebra $\Q_p \B \cong \Q_p \oplus \Q_p \oplus \Q_p$ splits, and the zeta function of $\Z_p \B$ can be examined using the techniques of \cite{BH2022}. 

When $v$ is not a square modulo $p$, $\Q_p \B \cong \Q_p \oplus K$, where $K$ is an unramified quadratic extension of $\Q_p$. If we let $e_0$ and $e_1$ be the primitive idempotents of $\Q_p\B$, then the maximal order of $\Q_p\B$ is $\Lambda_0=\langle \sqrt{v}e_1, e_1, e_0\rangle$, and we can represent the Hermite normal form of  $\Z_p\B$ with respect to this basis by $\left(\begin{smallmatrix} p^m & 0 & 0 \\ 0 & 1 & 1 \\ 0 & 0 & p^{2m}\end{smallmatrix}\right)$. Therefore, the calculation of the local $p$-zeta function of $\Z_p\B$ can be obtained using the same ideas as Theorem \ref{rank3}, Lemma \ref{lattice-lemma} and Theorem \ref{rank3highpower}.

\section{Zeta functions of Small Fusion Rings}

The most widely studied examples of fusion rings are those that arise as the Grothendieck rings of monoidal tensor categories; these are called {\it categorifiable} fusion rings (see \cite{EGNO}).   In these representation categories there are finitely many irreducible representations that play the role of basis elements.  The direct sum of representations is naturally a representation of the algebra, and the presence of a compatible co-algebra structure implies that the tensor product of representations also gives a representation.  So the nonnegative integer span of the irreducibles is closed under the tensor product, and using similarity of representations as an equality, the integer span of the irreducibles produces a fusion ring with the irreducible representations as the defining basis.   We will consider the zeta functions for the following categorifiable fusion rings of ranks $2$ and $3$: 

(Rank 2:) 

the Fibonacci fusion ring: $\B = \{1,b\}$, $b^2 = 1 + b$; and 

the group $C_2$; 

(Rank 3)
 
(a) the Ising fusion ring: $\B = \{1,b,d\}, b^2=1, bd=d, d^2=1+b$; 

(b) $Rep(S_3)$: $\B=\{1,b,d\}, b^2=1, bd=d, d^2=1+b+d$; 

(c) $PSU(5)$ level $2$: $\B=\{1,b,d\}, b^2=1+d, bd=b+d, d^2=1+b+d$;   

(d) the $E_6$-subfactor fusion ring: $\B=\{1,b,d\}, b^2=1, bd=d, d^2=1+b+2d$; and 

(e) the cyclic group $C_3$.

\smallskip
The Fibonacci fusion ring is the only categorifiable one of dimension $2$ that is not a group.  The five types of $3$-dimensional fusion categories corresponding to pivotal fusion categories were recently classified by Ostrik in \cite{Ostrik2015} - the pivotal assumption on these is believed to be unnecessary.  For dimension $4$ and higher complete classification remains open, but a considerable amount of recent work on special cases has appeared - see \cite{Larson2013}, \cite{DongChenWang2022}, \cite{DongZhangDai2018}, \cite{Bruillard2016}, and \cite{BruillardOrtiz2022}.  A list of small multiplicity-free (i.e.~commutative) fusion rings is under development at 
{\footnotesize{\tt https://anyonwiki.github.io/pages/Lists/losmffr.html}} \cite{VS}.

\begin{exa} {\rm 
{\bf The Fibonacci fusion ring.} The integral basis of the Fibonacci fusion ring is  $\B = \{1,b\}$, $b^2 = 1 + b$.  The largest eigenvalue of $b$ is $\delta(b)=\frac{1+\sqrt{5}}{2}$, the golden ratio.  Since $\Z \B = \Z[b] \simeq \Z [\frac{1+\sqrt{5}}{2}]$, and this is the ring of integers of $\Q \B \simeq \Q (\sqrt{5})$, we have that $\zeta_{\Z \B}(s)$ is equal to the Dedekind zeta function for the ring $\Z [\frac{1+\sqrt{5}}{2}]$.  
}\end{exa}

\medskip 
\begin{exa}  {\rm {\bf 3a. The Ising fusion ring.} For the Ising fusion ring, the standard basis will be $\B^s = \{1, b, \sqrt{2}d \}$.  The standard integral table algebra $\Q \B^s$ agrees with the adjacency algebra of the association scheme of order $4$ corresponding to the square.  The character table of this standard integral table algebra (with columns corresponding to elements of $\B^s$, rows to the irreducible characters whose multiplicities appear on the right) is 
$$ P = \begin{bmatrix} 1 & 1 & 2 \\ 1 & 1 & -2 \\ 1 & -1 & 0 \end{bmatrix} \begin{array}{l} 1 \\ 1 \\ 2 \end{array}. $$ 
Since the basis element $d \in \B$ has a quadratic irrational Perron-Frobenius eigenvalue $\sqrt{2}$, the degree character $\delta$ of $\B$ has a Galois conjugate $\delta^{\tau}$, and the third irreducible character $\phi$ will be rational-valued.  So we have that the primitive idempotents of $\C \B$ are 
$$ e_{\delta} = \frac{1}{4}(1+b+\sqrt{2}d), e_{\delta^{\tau}} = \frac{1}{4}(1+b-\sqrt{2}d), \mbox{ and } e_{\phi} = \frac{1}{2}(1-b).$$ 
The primitive idempotents of $\Q \B$ are $e_0=e_{\phi}$ and $e_1 = e_{\delta}+e_{\delta^{\tau}} = \frac{1}{2}(1 + b)$.  Note that $de_0=0$.  We have $\Q \B \simeq \Q \oplus \Q (\sqrt{2})$, and the maximal order is $\Lambda_0 = \langle de_1 = d,e_1,  e_0 \rangle$.  When we write $\B$ in terms of this basis, we get $1= e_0+e_1$, $b=e_1-e_0$, and $d = de_1$, which gives $\Z \B = \langle  de_1, e_1+e_0, 2e_0 \rangle$.  So $[\Lambda_0 : \Z \B]=2$ and we only need to calculate the local zeta function at $p=2$. 

 When we complete at $p=2$, the maximal order is $\Lambda_{0,2} \simeq  \Z_2 \oplus \Z_2[\sqrt{2}]$ and $\Z_2 \B = \{ x d e_1+ ye_1+ (y+2z)e_0 : x,y,z \in \Z_2 \}.$ Therefore, we can represent $\Z_2 \B$ as 
 $$\Z_2 \B \cong \{ (a, b)\in  \Z_2 \oplus \Z_2[\sqrt{2}]: b \pmod{2} \equiv a \pmod{\sqrt{2}} \}.$$
 Since  $\sqrt{2}$ is a uniformizer in the local ring $\Z_2[\sqrt{2}]$ with residue class field $\mathbb{F}_2$, the same local calculations as in \cite[Section 3.4]{Bushnell-Reiner1980} and Theorem \ref{rank3} apply, and 
 $$\zeta_{\Z_2 \B}(s) = (2^{1-2s}-2^{-s}+1) \zeta_{\Z_2}(s)^2.$$

Therefore, 
$$ \zeta_{\Z \B}(s) = (2^{1-2s}-2^{-s}+1) \zeta_{\Z}(s) \zeta_{\Z[\sqrt{2}]}(s). $$

\medskip
{\bf 3b. $Rep(S_3)$:} This time the standard basis will be $\B^s = \{1,b,2d\}$.  This the standard basis of an integral table algebra of order $6$ whose character table is 
$$ P = \begin{bmatrix} 1 & 1 & 4 \\ 1 & 1 & -2 \\ 1 & -1 & 0 \end{bmatrix} \begin{array}{l} 1 \\ 2 \\ 3 \end{array}. $$ 
This is the character table of the imprimitive association scheme corresponding to $K_2^{(3)}$ and its multipartite complement $K_{3 \times 2}$; it is dual to the conjugacy class scheme of $S_3$.   The primitive idempotents of $\Q \B$ are: 
$$ e_{\delta} = \frac{1}{6}(1+b+2d), e_{\psi} = \frac{1}{3}(1 + b - d), \mbox{ and } e_{\phi}=\frac{1}{2}(1-b), $$ 
so $b = e_{\delta}+e_{\psi}-e_{\phi}$ and $d = 2e_{\delta}-e_{\psi}$.  Simplifying the basis $\{1, b, d\}$, it follows that $\Z \B = \Z \cdot 1 + 2\Z (e_{\delta}+e_{\psi}) + 3\Z e_{\psi}$, so it has nontrivial completions at the primes $p=2$ and $p=3$.  The $2$- and $3$- local zeta functions of this order match those of of $\Z_p K_p \oplus \Z_p$, $p=2,3$, where $K_p$ is the complete graph association scheme on $p$ vertices.  The zeta function of $\Z_p K_p$ has been worked out for all primes $p$ in \cite{Hanaki-Hirasaka2016} and \cite{BH2022}.   From the formula in \cite{BH2022} we get 
$$ \zeta_{\Z \B}(s) = (2^{1-2s} - 2^{-s} + 1)(3^{1-2s} - 3^{-s} + 1) \zeta_{\Z}(s)^3. $$

\medskip
{\bf 3c. $PSU(5)$ level $2$:} Note that $b = d^2-d-1$.  This implies $\Z \B = \mathbb{Z}[d]$, and $d$ has minimal polynomial $x^3-2x^2-x+1$.  The roots of this polynomial are the Galois conjugates  of $\alpha = -\zeta_7-\zeta_7^2-\zeta_7^5-\zeta_7^6$.  As this element of the ring of integers of the totally real cubic subfield $\mathbb{Q}(\zeta_7)^+$  of  $\mathbb{Q}(\zeta_7)$ has norm $1$, $\Z[\alpha]$ is equal to the ring of integers of $\mathbb{Q}(\zeta_7)^+$.  Therefore, $\zeta_{\Z \B}(s)$ will be equal to the Dedekind zeta function of $\Z[\alpha]$. 

\medskip 
{\bf 3d. The rank $3$ $E_6$ subfactor fusion ring:} This time $b=d^2-2d-1$, so $\Z \B = \Z[d]$.  Since $\delta(b)=1$, we have $\delta(d)^2=2+2\delta(d)$, which means $\delta(d) = 1 + \sqrt{3}$.  
The minimal polynomial of $d$ is $x^3-2x^2-2x$, so $\Q \B \simeq \Q \oplus  \Q[1+\sqrt{3}]$.  The primitive idempotents of $\Q \B$ are $e_0 = \frac12 (1-b)$ and $e_1=\frac12 (1 + b)$, and the maximal order is $\Lambda = \langle de_1=d, e_1, e_0\rangle$.  Since $1 = e_0+e_1$, $b = e_1 - e_0$, and $d = de_0$, we get $\Z\B = \langle de_1, e_1+e_0, 2e_0 \rangle$.  So we only need the local zeta function at $p=2$.   
The completion at $p=2$ of the maximal order of $\Q_2 \B$ is $\Lambda_0 = \Z_2 e_0 \oplus \Z_2\B e_1 \simeq  \Z_2 \oplus \Z_2[1+\sqrt{3}]$, and $\pi = 1+\sqrt{3}$ is a uniformizer of $\Z_2[1+\sqrt{3}]$ of norm $-2$. Then $\Z_2 \B$ corresponds to $\{(\alpha, x+y\pi) \in \Lambda_0 :  a \mod 2 \equiv x+y\pi \mod \pi  \}$, and this condition is equivalent to $x \equiv \alpha \mod 2$, $y \in \Z_2$. 
Again we have just two genera of $\Z_2 \B$-lattices, which are represented by $\Lambda_0$ and $\Z_2 \B$.  The calculation of the local zeta function in this case matches that of the Ising fusion ring in case 3a, so in the end we will get: 
$$ \zeta_{\Z \B}(s) = (2^{1-2s}-2^{-s}+1) \zeta_{\Z}(s) \zeta_{\Z[\sqrt{3}]}(s).$$ }

\medskip
{\bf 3e. The cyclic group $C_3$:} Solomon computed $\zeta_{\Z C_p}(s)$ for $p$ prime in \cite{Solomon77}, the result matches that of the asymmetric rank $3$ table algebra discussed in the previous section, in the case where the order is $3$: 
$$\zeta_{\Z C_3}(s) = (3^{1-2s}-3^{-s}+1)\zeta_{\Z}(s)\zeta_{\Z[\sqrt{-3}]}(s).$$  
\end{exa}

\end{document}